\documentclass{gtpart}     % Basic GT/GTM/AGT style
\usepackage[all, cmtip]{xy}
\usepackage{graphicx, subfig}
\usepackage{epsfig}
\usepackage{amscd}
\usepackage[mathscr]{eucal}
\usepackage{amssymb}
\usepackage{amsxtra}
\usepackage{amsmath}
\usepackage{latexsym}
\usepackage[all]{xy}
\usepackage{enumerate}
\usepackage{mathrsfs}
\usepackage{color}
\title[Gromov width of surfaces]{Gromov width and uniruling for orientable Lagrangian surfaces}

\author{Fran\c{c}ois Charette}
\givenname{}
\surname{}
\address{Fran\c{c}ois Charette, Department of Mathematics, ETH-Z\"{u}rich,
  R\"{a}mistrasse 101, 8092 Z\"{u}rich, Switzerland}
\email{francois.charette@math.ethz.ch}
\urladdr{}


\arxivreference{http://arxiv.org/abs/1401.1953}
\arxivpassword{bsxbu}

\volumenumber{}
\issuenumber{}
\publicationyear{}
\papernumber{}
\startpage{}
\endpage{}
\doi{}
\MR{}
\Zbl{}
\received{}
\revised{}
\accepted{}
\published{}
\publishedonline{}
\proposed{}
\seconded{}
\corresponding{}
\editor{}
\version{}

\newtheorem{thm}{Theorem}[section]    % Standard theorem environment
\newtheorem{lem}[thm]{Lemma}          % Lemma environment with numbering 
\theoremstyle{definition}
\newtheorem{mainthm}{Theorem}
\newtheorem{prop}[thm]{Proposition}
\theoremstyle{remark}
\newtheorem{rem}[thm]{Remark}
\newtheorem*{remnonum}{Remark}

\newcommand{\id}{\textnormal{id}}

\newcommand{\Mm}{{\mathcal{M}}}

\newcommand{\Crit}{\textnormal{Crit\/}}
\newcommand{\jreg}{\mathcal{J}_\text{reg}}

\newcommand{\ind}{{\textnormal{ind}}}

\begin{document}
\begin{abstract}    % type your abstract below
We prove a conjecture of Barraud-Cornea for orientable Lagrangian surfaces. As a corollary, we obtain that displaceable Lagrangian 2--tori have finite Gromov width.  In order to do so, we adapt the pearl complex of Biran-Cornea to the non-monotone situation based on index restrictions for holomorphic discs.
\end{abstract}

\maketitle

%%%%%%%%%%%%%%%%%%%%   Start of main body of article
% !TEX root = Serep.tex

\section{Introduction}
The present paper is a continuation of results of the author \cite{Cha:refinement}, where it was shown that closed \textit{monotone} (see \S \ref{sec:prelim} for the precise defintion) Lagrangians in tame symplectic manifolds satisfy a general form of uniruling by holomorphic curves. Here we remove the monotonicity condition in the case of orientable surfaces.  Recent examples given by Rizell \cite{Riz:nonuniruled} show that this restriction is not of a technical nature; the results simply do not hold in higher dimensions without other constraints, or even for non-orientable surfaces. In this note we focus our attention mostly on the connection between displacement energy and Gromov width.

Recall that the relative Gromov width of a Lagrangian $L$ is 
$$w_G(L):= \sup_{\mathcal{B}(M, L, r)} \{ \pi r^2 \},$$
where $\mathcal{B}(M, L, r)$ is the set of all symplectic embeddings of $B^{2n}(r)$, the ball of radius $r$ in $\C^n$, such that $B^{2n}(r) \cap \R^n$ is mapped to $L$.

The displacement energy of a Lagrangian is the minimal energy required by a Hamiltonian isotopy to displace it, $E(L):= \inf \{ E(\phi) \; | \; \phi \in \text{Ham}(M, \omega), \; \phi(L) \cap L = \emptyset \}$, where $E(\phi)$ is the energy of a Hamiltonian isotopy; we set $E(L) = \infty$ in case it is not displaceable.

If $L$ is a closed, displaceable and orientable Lagrangian surface, an easy argument shows that it is diffeomorphic to a torus.  Our main result is the following
\begin{mainthm} \label{thm:main}
Let $L$ be a  Lagrangian 2--torus in a tame symplectic manifold, then $w_G(L) \leq 2 E(L)$.
\end{mainthm}
The proof is based on Theorem \ref{thm:uniruling}, a uniruling result which in turns proves a conjecture of Barraud-Cornea \cite{Bar-Cor:NATO} for orientable surfaces.  

There are other non-monotone situations where the Gromov width of a displaceable Lagrangian is known to be finite.  Recent results of Borman-McLean \cite{BorMcL:widths} show that if $L$ is orientable and admits a metric of non-positive sectional curvature in a Liouville manifold, then its Gromov width is bounded above by four times its displacement energy.

Uniruling for monotone Lagrangians is established via a mixture of the pearl complex from Biran-Cornea \cite{Bi-Co:Yasha-fest,Bi-Co:rigidity,Bi-Co:qrel-long} and Lagrangian Floer theory of Floer \cite{Fl:Morse-theory}, Oh \cite{Oh:HF1,Oh:HF1-add} and Fukaya-Oh-Ohta-Ono \cite{FO3:book-vol1,FO3:book-vol2}.  We show how these constructions can be adapted to our case, with an appropriate choice of Novikov ring.  The key argument is an elementary index computation of pseudo-holomorphic discs given in Lemma \ref{lem:index}, which itself relies on a technical result of Lazzarini \cite{Laz:decomp}.  There is no need to invoke the cluster complex of Cornea-Lalonde \cite{Cor-La:Cluster-1} or the general Lagrangian Floer theory of \cite{FO3:book-vol1}; both these theories are much more complicated algebraically and even more so analytically.

We summarize in the next proposition the algebraic structures that we define for orientable surfaces by adapting the techniques of Biran-Cornea, without any monotonicity assumptions.  The Novikov ring $\Lambda$ is defined in \S \ref{sec:pearl}.
\begin{prop} \label{thm:mainstructures}
 Given a closed orientable Lagrangian surface $L \subset (M, \omega)$, there is a second category subset $\jreg \subset \mathcal{J}_\omega$ of regular compatible almost complex structures for which the following algebraic structures are defined and depend only on the connected component $[J] \in \pi_0(\jreg)$.
  \begin{itemize}
  \item The Lagrangian quantum homology ring $QH(L, [J]; \Lambda)$ of $L$, endowed with the quantum product;
  \item The Lagrangian Floer homology $HF(L, H, \{ J_t \}; \Lambda)$ of $L$, where $\mathbf{J}:= \{ J_t \}$ is a generic path of regular almost complex structures and $H$ is a Hamiltonian.  It is a left $QH(L, [J_0]; \Lambda)$-module and a right $QH(L, [J_1]; \Lambda)$-module;
  \item The $QH(L, [J_0])$-module isomorphism $PSS\co QH(L, [J_0]) \to HF(L, H, \mathbf{J})$
 \end{itemize}
\end{prop}

Details of the construction are given mostly for the Lagrangian quantum homology in \S \ref{sec:lagqh}, as this should emphasize the main ideas;  the other structures are only quickly sketched in \S \ref{sec:lagfloermodule}.  The module property of the PSS isomorphism is adapted from  the author \cite{Cha:refinement} and was first proved by Leclercq \cite{Leclercq:spectral} when there is no bubbling. Finally, a uniruling theorem is given in \S \ref{sec:proof}, from which we  deduce Theorem \ref{thm:main}.

\begin{rem}\label{rem:Paul}
Assume that $M$ is closed and denote by $\beta_2^+(M)$ the dimension of the positive definite part of the intersection form on $H_2(M)$.  If $L$ is a closed orientable Lagrangian surface of genus $g$ at least 2, then $\beta_2^+(M) \geq 2$.  Indeed $\omega^2 >  0$ and $[L]^2 = -\chi(L) = 2g -2 >0$.  Also, $[L]$ is linearly independent from $\omega$, so that $\beta_2^+(M) \geq 2$.

Much is known about symplectic 4-manifolds and the value of $\beta_2^+(M)$, see for example Baldridge \cite{Bald:b2plus},  Gompf-Stipsicz \cite{Gom-Stip:kirby} and McDuff-Salamon \cite{McD-Sa:b2plus}.
\end{rem}

\

\noindent {\bf Acknowledgements.}  I would like to thank Octav Cornea for encouraging me to write this article, for many useful explanations and comments on early drafts.  I thank Paul Biran for bringing Remark \ref{rem:Paul} to my attention.  The author was supported by an ETH Z\"urich fellowship.

% !TEX root = lagsurf.tex

\section{Algebraic structures and the proof of Proposition \ref{thm:mainstructures}}\label{sec:structures}

\subsection{Preliminaries}\label{sec:prelim}
Throughout the paper, $(M, \omega)$ is a tame symplectic four dimensional manifold.  The set of all $\omega$-compatible almost complex structures is denoted by $\mathcal{J}_\omega$, the induced Riemannian metric is $g(v_1, v_2):= \omega(v_1, J v_2)$ and the associated first Chern class is written as $c_1(M)$ or $c_1$.

The Hamiltonian vector field $X_{H_t}$ of a compactly supported $H\co S^1 \times M \to \R$ is uniquely defined by $\omega(X_{H_t}, \cdot) = -dH_t(\cdot)$.  Its time one flow $\psi_1^H$ defines a Hamiltonian isotopy and the set of all these isotopies is the group $\text{Ham}(M, \omega)$.  The energy of $\psi \in \text{Ham}$ is defined as 
$$E(\psi) = \inf_{H \; | \; \psi_1^H = \psi} \int_{S^1} (\max_M H_t - \min_M H_t)dt.$$

Let $L$ be a closed Lagrangian surface in $M$, then there are two morphisms
$$\omega\co H_2(M, L) \to \R, \quad \mu\co H_2(M, L) \to \Z$$
given by the symplectic area and the Maslov index.  Recall that this index is defined for every $u\co (\Sigma, \partial \Sigma) \to (M, L)$, where $\Sigma$ is a surface with boundary, and that it depends only on the homology class of $u$.  Finally, we say that $L$ is monotone if two conditions hold:
\begin{enumerate}
  \item There exists $\rho > 0$ such that $\omega = \rho \mu$;
  \item The positive generator of the image of $\mu$, denoted by $N_L$, is greater or equal than two.
\end{enumerate}

In this paper, \textbf{no monotonicity conditions} are imposed on Lagrangians.
\subsection{Lagrangian quantum homology}\label{sec:lagqh}

%%%%%%%%%%%%%%%%%%%%%%%%%%%%%%%%%%%%%%%%%%%%%%%%%%%%%%%%%%%%%%%%%%%%%%%%%%%%
\subsubsection{Index restriction for $J$-holomorphic discs and spheres; counting discs of Maslov class two}\label{sec:discinv}
Recall that a disc $u\co (D^2, S^1) \to (M, L)$ is somewhere injective if there exists $z \in D^2$ such that 
$$u^{-1}(u(z)) = z, \; du(z) \neq 0,$$
and that $u$ is \textit{simple} if the set of injective points is dense.

The space of simple $J$-holomorphic discs in a class $A \in H_2(M, L)$ is denoted by $\Mm^*(A;J)$.  For an $A$-regular $J$, i.e. one for which the linearization $D_u$ of the $\bar{\partial}_J$ operator at $u$ is onto, this space is a manifold of dimension 
\begin{equation}\label{eq:dim}
 \dim \Mm^*(A;J) = n + \mu(A) = 2 + \mu(A). 
\end{equation}
We denote by $\jreg(A) \subset \mathcal{J}_\omega$ the second category subset of $A$-regular almost complex structures and define $\jreg := \bigcap_A \jreg(A)$.  Since $H_2(M, L)$ is countable, $\jreg$ is also of second category.

It is well-known that $J$-holomorphic spheres which are not simple are multiply covered.  For non simple discs, the same holds when the dimension of the Lagrangian is at least 3 and $J$ is generic, by a theorem of Lazzarini \cite{Laz:decomp}.  However, in dimension two the situation is quite different:
\begin{thm}[Theorem A in \cite{Laz:decomp}]\label{thm:decomp}
 Let $u\co (D^2, S^1) \to (M, L)$ be a non constant $J$-holomorphic disc.  Then there exists a finite family $\{ v_i \}$ of simple $J$-holomorphic discs $v_i\co (D^2, S^1) \to (M, L)$ and positive integers $\{ m_i \}$ such that, in $H_2(M, L)$, we have $[u] = \sum_i m_i [v_i]$ and $\bigcup_i v_i(D^2) = u(D^2)$.
\end{thm}
\noindent Here $[u]$ denotes the image of the positive generator of $\Z \cong H_2(D^2, S^1)$ by $u_*$.

The following elementary lemma is the key to the construction of all the algebraic structures considered in Proposition \ref{thm:mainstructures}.
\begin{lem}\label{lem:index}
  Let $L$ be an orientable Lagrangian surface, $J \in \jreg$ a regular almost complex structure and $u\co (D^2, S^1) \to (M, L)$ a non-constant $J$-holomorphic disc, then $\mu(u) \geq 2$.  If $u$ is a non-constant sphere, then $c_1(u) \geq 1$. In both cases $u$ is simple if equality holds.
\end{lem}
\begin{proof}
  First, given any orientable Lagrangian, regardless of dimension, the Maslov index of any disc $u$ with boundary on $L$ is even, since the loop $(u|_{S^1})^*  TL \in \mathcal{L}(\R^{2n})$ lifts to the double cover $\mathcal{L}^{or}(\R^{2n})$ of oriented Lagrangian subspaces.
  
  Using the notations of Theorem \ref{thm:decomp}, we have $\mu(u) = \sum_i m_i \mu(v_i)$, where all $v_i$'s are non constant simple $J$-holomorphic discs. Thus biholomorphisms of the disc, denoted by $\mathcal{G}:= PSL(2, \R) = Aut(D^2) \cong S^1 \times D^2$, act freely on $\Mm^*([v_i]; J)$ and equation (\ref{eq:dim}) shows that 
  $$0 \leq \dim \Mm^*([v_i]; J)/\mathcal{G} = 2 + \mu(v_i) - 3 = \mu(v_i) - 1 \iff
  2 \leq \mu(v_i).$$
  Hence we get that $\mu(u) = \sum_i m_i \mu(v_i) \geq 2.$
  
  As for the last part, $u$ can be written as $u = v \circ d_k$, where $v\co S^2 \to M$ is simple and $d_k\co S^2 \to S^2$ is a degree $k \geq 1$ holomorphic covering map.  Then
  $$0 \leq \dim \Mm^*([v]; J)/PSL(2,\C) = 4 + 2c_1(v) - 6,$$
  and $c_1(u) = kc_1(v) \geq k$.
\end{proof}

Fix $J \in \jreg$ and $A \in H_2(M, L)$ such that $\mu(A) =2$, then there is an evaluation map
$$ev\co  \mathcal{M}^*(A; J) \times S^1 / \mathcal{G} \to L$$
$$[u, \theta] \mapsto u(\theta).$$ 
The action of $g \in \mathcal{G}$ is given by $g \cdot (u, \theta) = (u \circ g, g^{-1}(\theta))$.  The index computations above and Gromov compactness for discs (see Frauenfelder \cite{Fr:msc}) yield that $\mathcal{M}^*(A; J) \times S^1/\mathcal{G}$ is a closed 2-dimensional manifold, hence $ev$ has a well-defined mod 2 degree, giving the algebraic count of $J$-discs representing $A$ and going through a generic point of $L$, denoted by $d(A, J)$.  However, this invariant might depend on the choice of $J$.  Indeed, fix a path $\mathbf{J}:= \{ J_t \}, \; J_t \in \mathcal{J}_\omega$ connecting $J_0$ and $J_1$.  Then for a generic choice of such path (see McDuff-Salamon \cite[Definition 3.1.6]{McD-Sa:Jhol-2}), the space 
$$\mathcal{W}^*(A, \mathbf{J}):= \{ (t, u) \; | \; t \in [0,1], \; u \in \mathcal{M}^*(A, J_t)\}$$
provides a cobordism between $\mathcal{M}^*(A, J_0)$ and $\mathcal{M}^*(A, J_1)$.  The next lemma is elementary, but it is worth mentioning since it is the possible source of non-invariance for the Lagrangian quantum homology of orientable surfaces, see \S \ref{sec:pearl}.

\begin{lem}\label{lem:nonreg}
 Fix a generic path $\mathbf{J}$.  If it is in a connected component of $\jreg$, called a chamber, then $\mathcal{W}^*(A, \mathbf{J})$ is compact, hence $d(A, J_0) = d(A, J_1)$.  Otherwise, it admits a compactification which must include non-constant Maslov zero $J_t$-holomorphic discs for every $J_t \notin \jreg$.
\end{lem}
\begin{proof} 
By Gromov compactness, sequences in $\mathcal{W}^*(A, \mathbf{J})$ converge to stable maps $(T, \sum A_k := A)$ where $T$ is a tree made up of $J_t$-holomorphic bubbles for some fixed $t$; each bubble is either a sphere or a disc $u_k$ representing a class $A_k$.  Moreover, their Maslov indices satisfy $\sum \mu(u_k) = \mu(A)$.  We use the convention that for spheres, the Maslov index is twice the Chern class.
 
Assume first that $\mathbf{J} \subset \jreg$.  Then Lemma \ref{lem:index} implies that there is only one curve in $T$ and the bubble tree is made of a single simple curve, hence it is already an element of $\mathcal{W}^*(A, \mathbf{J})$, which is thus compact.  Since degree is invariant under compact cobordisms, we get that $d(A, J_0) = d(A, J_1)$.

Assume now that $\mathbf{J}$ is not entirely contained in $\jreg$.  Then, by definition of regularity for the path $\mathbf{J}$, we have $\dim \text{coker } D_{u_k} = 1$, since we assume $J_t \notin \jreg$.
 
 If $u_k$ is a sphere, then since it is not constant, $PSL(2, \C)$ acts on it and $\dim \ker D_{u_k} \geq 6$.  Thus
 $$\dim \ker D_{u_k} - \dim \text{coker } D_{u_k} = \ind \; D_{u_k} = 2n + 2c_1(u_k) = 4 + 2c_1(u_k) \geq 6 - 1 = 5,$$
 so $2c_1(u_k) = \mu(u_k) \geq 2$. 
 
 If $u_k$ is a non-constant disc bubble, then $\dim \ker D_{u_k} \geq 3$ and $\dim \text{coker} D_{u_k} = 1$, hence
 $$\ind \; D_{u_k} = n + \mu(u_k) = 2 + \mu(u_k) \geq 3 - 1 = 2 \iff \mu(u_k) \geq 0.$$

Finally, since the total Maslov class of the tree is $\mu(A)=2$, the bubbles have index at most two and this concludes the proof.
\end{proof}

\begin{remnonum}
 Examples of Maslov zero non-regular $J$-holomorphic discs in Lagrangian tori can be found in Auroux \cite{Aur:t-duality}.
\end{remnonum}

\subsubsection{The pearl complex}\label{sec:pearl}
In this section we adapt the construction of the pearl complex following the presentation of Biran-Cornea \cite{Bi-Co:qrel-long} closely.  Lemma \ref{lem:index} and our choice of Novikov ring guarantee that their original arguments are still valid.

Let $J \in \jreg$ and $f\co L \to \R$ be a Morse-Smale function with respect to a generic Riemannian metric $\rho$, and denote by $\phi_t$ the \textit{negative} gradient flow of $(f, \rho)$.  The universal Novikov field is
$$\Lambda^{\text{univ}} := \{ \sum_k a_k T^{\lambda_k} \;| \; a_k \in \Z_2, \; \lambda_k \in \R, \; \lim_{k \to \infty} \lambda_k = \infty\}.$$
We set $\Lambda:= \Lambda^{\text{univ}}[q, q^{-1}]$ and grade it with $|q| = -1$.

The pearl complex of $f$ is the $\Lambda$-module
$$\mathcal{C}(f, J, \rho):= \text{span}_\Lambda \langle\Crit f\rangle.$$

The differential counts pearly trajectories which we now describe.  Given $x, y \in \Crit f$ and $A \in H_2(M, L)$, the space of pearls from $x$ to $y$ in the class $A$ is denoted by $\mathcal{P}(x, y, A; f, J, \rho)$ and consists of families $(u_1, ..., u_k)/ (\oplus_{i=1}^k \mathcal{G}_{-1, 1})$ such that
\begin{itemize}
 \item $u_i\co (D^2, S^1) \to (M, L)$ are non constant $J$-holomorphic discs.
 \item $u_1(-1) \in W^u(x)$.
 \item $\exists \; 0 < t_i < \infty$ such that $\phi_{t_i}(u_i(1)) = u_{i+1}(-1),$ for $i=1, ..., k-1$.
 \item $u_k(1) \in W^s(y)$.
 \item $\sum_i [u_i] = A$.
 \item $\mathcal{G}_{-1,1}$ is the subgroup of elements of $\text{Aut}(D^2)$ fixing $\pm 1$ and the direct sum of these groups act on the family $(u_1, ..., u_k)$.
\end{itemize}
Whenever $A=0$, a pearl means a negative gradient flow line from $x$ to $y$, modulo the $\R$-action.

Denote $\mathcal{P}^{*,d}(x, y, A; f, J, \rho)$ the subspace of pearls for which all the discs are simple and absolutely distinct.  A standard transversality argument shows that $\mathcal{P}^{*,d}(x, y, A; f, J, \rho)$ is either empty or a manifold of dimension $|x| - |y| + \mu(A) - 1$.  

%In order to show that the differential squares to zero, we introduce the space of broken pearls.  Fix two homology classes $A, B \in H_2(M, L)$, $x,y \in \Crit f$ and define $\mathcal{P}(x, y, A, B; f, J, \rho)$ to be the set of families $(u_1, ..., u_k, v_1, ..., v_r) / (\oplus_{i=1}^{k+r} \mathcal{G}_{-1. +1})$ such that 
%\begin{itemize}
% \item $u_i: (D^2, S^1) \to (M, L)$ and $v_i: (D^2, S^1) \to (M, L)$ are non constant $J$-holomorphic discs.
% \item $u_1(-1) \in W^u(x)$.
% \item $\exists \; 0 < t_i < \infty$ such that $\phi_{t_i}(u_i(1)) = u_{i+1}(-1),$ for $i=1, ..., k-1$.
% \item $u_k(1) = v_1(-1)$.
% \item $\exists \; 0 < s_i < \infty$ such that $\phi_{s_i}(v_i(1)) = v_{i+1}(-1),$ for $i=1, ..., r-1$.
% \item $v_r(1) \in W^s(y)$.
% \item $\sum_i [u_i] = A, \; \sum [v_i] = B$.
%\end{itemize}
%Those are pearls where exactly one flow line is constant between two discs.

%Denote $\mathcal{P}^{*,d}(x, y, A, B; f, J, \rho)$ the subspace of simple and absolutely distinct broken pearls.  Again, standard arguments show that it is either empty or a manifold of dimension $|x| - |y| + \mu(A) + \mu(B) - 2$.

In \cite[\S 3]{Bi-Co:rigidity} Biran and Cornea prove the following crucial result for \emph{monotone} Lagrangians, needed to show that pearls can be used to define a differential, which itself relies on Theorem \ref{thm:decomp} of Lazzarini.

\begin{prop}[\cite{Bi-Co:rigidity}]\label{prop:pearls}
  Let $L$ be a monotone Lagrangian and $f, \rho$ be as defined above.  Then there exists a second category subset $\jreg \subset \mathcal{J}_\omega$ with the following property.  For every $A \in H_2(M, L)$ and every $x, y \in \Crit f$ such that $|x| - |y| +\mu(A) - 1 \leq 1$, we have:
  \begin{enumerate}
   \item $\mathcal{P}^{*,d}(x, y, A; f, J, \rho) = \mathcal{P}(x, y, A; f, J, \rho)$, i.e. all pearls are automatically simple and absolutely distinct.
   \item If $|x| - |y| + \mu(A) - 1 = 0$, then $\mathcal{P}(x, y, A; f, J, \rho)$ is a compact 0-dimensional manifold, and hence consists of a finite number of points.
  \end{enumerate}
\end{prop}

First note that Lazzarini's result holds for Lagrangians which are not monotone.  Also, a careful inspection of Biran and Cornea's proof shows that, as long as the Maslov index of J-holomorphic discs is at least two, then Proposition \ref{prop:pearls} is true, without the monotonicity assumption.  By Lemma \ref{lem:index}, this condition is automatically satisfied for closed orientable surfaces, hence
%\begin{prop}[Proposition 3.1.6 in \cite{Bi-Co:qrel-long}]\label{prop:brokenpearls}
% Let $L, f, \rho$ be as above.  Let $A, B \in H_2(M, L)$ be such that $|x| - |y| + \mu(A) + \mu(B) -1 \leq 1$. Then there exists a second category subset $\jreg \subset \mathcal{J}_\omega$ such that:
% \begin{enumerate}
%  \item $\mathcal{P}^{*,d}(x, y, A,B; f, J, \rho) = \mathcal{P}(x, y, A,B; f, J, \rho)$
%  \item If $|x| - |y| + \mu(A) + \mu(B) -1 \leq 0$, then $\mathcal{P}(x, y, A,B; f, J, \rho)$ is empty.
%  \item If $|x| - |y| + \mu(A) + \mu(B) -1 = 1$, then $\mathcal{P}(x, y, A,B; f, J, \rho)$ is a compact 0-dimensional manifold.
% \end{enumerate}
%\end{prop}
we can define the pearl differential
$$\partial\co \mathcal{C}_*(f, J, \rho) \to \mathcal{C}_{*-1}(f, J, \rho)$$
$$x \mapsto \sum_y (\sum_{\substack{A \in H_2(M, L) \\ |x| - |y| + \mu(A) -1 =0} } \#_2 \mathcal{P}(x,y, A; f, J, \rho)T^{\omega(A)}q^{\mu(A)})y$$
The coefficient of each $y$ is a well defined element of the ring $\Lambda$ by standard arguments.

\begin{rem}\label{rem:lagrestriction}
The arguments in this paper work also for Lagrangians whose minimal Maslov number is at least two when restricted to J-holomorphic discs, which include the class of monotone Lagrangians and, a fortiori, orientable surfaces.  It is however not clear to the author how to verify this condition in general.
\end{rem}

In order to show that $\partial^2 = 0$, one uses Lemma \ref{lem:index} to rule out side bubbling and conclude that one-dimensional family of pearls can be compactified using only broken flow lines.  We denote the resulting Lagrangian quantum homology by $QH(L, J; \Lambda)$ or $QH(L, J)$.

The quantum product can also be defined and makes Lagrangian quantum homology a ring with unity, as well as a left module over itself:
$$\circ\co QH_p(L, J; \Lambda) \otimes QH_q(L, J; \Lambda) \to QH_{p+q-2}(L, J; \Lambda).$$

Comparison morphisms 
$$\psi\co \mathcal{C}(f_1, J_1, \rho_1) \to \mathcal{C}(f_2, J_2, \rho_2)$$
are defined via Morse homotopies and regular paths of almost complex structures.  As was pointed out in Lemma \ref{lem:nonreg}, it is necessary to restrict to a regular path in a connected component of $\jreg$ to avoid non regular bubbling of Maslov zero discs.  In this case we directly get
\begin{prop}
 For $[J_1] = [J_2] \in \pi_0(\jreg)$, the comparison morphisms are chain maps which induce isomorphisms in homology and are compatible with the quantum product.
\end{prop}

%%%%%%%%%%%%%%%%%%%%%%%%%%%%%%%%%%%%%%%%%%%
\subsection{Lagrangian Floer homology and $QH(L, [J])$-module structure}\label{sec:lagfloermodule}
Fix a Hamiltonian $H\co S^1 \times M \to \R$ such that $(\phi^H_1)^{-1}(L)$ intersects $L$ transversally and denote by $\mathcal{O}(H)$ the finite set of contractible time-1 periodic orbits of the Hamiltonian flow $\phi_t^H$. 

For each half-disc $u$ contracting such an orbit $\gamma$, recall that there is a Maslov index $\mu(u, \gamma)$ such that $\mu(u, \gamma) + \frac{\dim L}{2} \in \Z$, see Robbin-Salamon \cite{Ro-Sa:spectral}.  There is an equivalence relation on tuples $(\gamma, u)$ defined by $(\gamma_1, u) \sim (\gamma_2, v) \iff \gamma_1 = \gamma_2$ and $\mu(\gamma_1, u) = \mu(\gamma_2, v)$.  For each $\gamma$, fix a class $\tilde{\gamma}:= [u_\gamma, \gamma]$ and grade it by $|\tilde{\gamma}| = 1 - \mu(\tilde{\gamma})$.

Denote by $\mathcal{J}^I := C^\infty([0,1], \mathcal{J}_\omega)$ smooth paths of compatible almost complex structures.  Given $\mathbf{J}:= \{ J_t \} \in \mathcal{J}^I$,  $\gamma_-, \gamma_+ \in \mathcal{O}(H)$, the set of Floer strips from $\tilde{\gamma}_-$ to $\tilde{\gamma}_+$ in a homology class $A \in H_2(M, L)$ is
$$\mathcal{M}(\tilde{\gamma}_-, \tilde{\gamma}_+, A; \mathbf{J}, H):=
\left\{ 
\begin{array}{c}
u\co \mathbb{R} \times [0,1] \to M,\\
u(\R, i) \in L, \; i=0,\;1\\
\end{array}
\middle|
\begin{array}{c}
\partial_s u  + J_t(u) (\partial_t u - X_{H_t}(u)) = 0\\
\lim_{s \to \pm \infty} u(s,t) = \gamma_\pm(t)\\
u_{\gamma^-} \# u  \# -u_{\gamma^+} = A
\end{array} 
\right\}.
$$
There exists a second category subset of regular paths $\jreg^I \subset \mathcal{J}^I$ for which the set of Floer strips are manifolds of dimension $|\tilde{\gamma}_-| - |\tilde{\gamma}_+| + \mu(A)$.  Recall that $\R$ acts freely on the space of non-constant strips, hence the quotient space is also a manifold, of dimension smaller by one.

We further restrict $\jreg^I$ to paths which are included in a chamber of $\jreg$ and pick such a path $\mathbf{J}$.  The Floer complex is the graded $\Lambda$-module defined by
$$CF(L; H, \mathbf{J}) := \text{span}_\Lambda \langle\tilde{\gamma} \; | \; \gamma \in \mathcal{O}(H)\rangle$$
and the differential is given on generators by
$$\partial\co CF_*(L; H, \mathbf{J}) \to CF_{*-1}(L; H, \mathbf{J})$$
$$\tilde{\gamma}_- \mapsto \sum_{\substack{\gamma_+, A\\ |\tilde{\gamma}_-| - |\tilde{\gamma}_+| + \mu(A) = 1}} \#_2 (\mathcal{M}(\tilde{\gamma}_-, \tilde{\gamma}_+, A; \mathbf{J}, H)/ \R) T^{\omega(A)}q^{\mu(A)}\tilde{\gamma}_+.$$
This sum is well-defined by our choice of $\mathbf{J}$ and by Gromov compactness.  Moreover, $\partial^2 = 0$ follows by standard arguments;  one basically mimics the proofs by Oh \cite{Oh:HF1,Oh:HF1-add}.  Of course we also need that for every $\gamma \in \mathcal{O}(H)$ and $A \in \mathbb{R}$, the algebraic numbers of $J_i$-holomorphic discs of Maslov class two and total area at most $A$ going through $\gamma(i)$ are equal, for $i=0, 1$.  This holds by our choice of $\mathbf{J} \subset \jreg$.

As with the pearl complex, if one restricts to a fixed connected component of $\jreg$, then the Floer homology is canonically independent of all choices.

The module structure is defined as in Charette \cite{Cha:refinement}:
$$\star\co \mathcal{C}_p(f, J_0) \otimes_\Lambda CF_q(L, H, \mathbf{J}) \to CF_{p+q-2}(L, H, \mathbf{J})$$
$$x \otimes \tilde{\gamma}_- \mapsto \sum_{\tilde{\gamma}_+, \; A \in H_2(M, L)} \#_2 \mathcal{M}(x, \tilde{\gamma}_-, \tilde{\gamma}_+, A) T^{\omega(A)} q^{\mu(A)} \tilde{\gamma}_+,$$
where $\mathcal{M}(x, \tilde{\gamma}_-, \tilde{\gamma}_+, A)$ is the space of pearls leaving $x$ and entering a Floer strip $u$ (at the point $u(0,0)$) connecting $\tilde{\gamma}_-$ to $\tilde{\gamma}_+$.  The homology class is $A = \sum_i [v_i] + [u_{\gamma_-} \# u \#  -u_{\gamma_+}]$, where $v_i$ are the discs appearing in the pearl from $x$ to $u$.  This gives $HF(L, \mathbf{J})$ a structure of left $QH(L, [J_0])$-module.

\subsubsection{The PSS isomorphism}\label{sec:PSS}
We quickly recall here the remaining structures mentioned in Proposition \ref{thm:mainstructures}; see \cite{Bi-Co:qrel-long, Cha:refinement} for more details and proofs.

The PSS-morphism
$$\mathcal{C}_*(f, J) \to CF_*(L, H, \mathbf{J})$$
is defined by counting pearls leaving from a critical point where the last disc in the pearl is a half-disc $u$ satisfying the PSS-equation and converging to an orbit $\tilde{\gamma}_+$.  We count these pearls in each homology class $A = \sum v_i + [u \# -u_{\gamma_+}]$.  The PSS-equation is
$$\partial_s u + J(s,t,u(s,t)) ( \partial_t u - \beta(s) X_{H_t}(u)) = 0,$$
where $\beta\co \R \to [0,1]$ is a monotone surjective function which is constant close to $\pm \infty$. An easy computation yields the following energy estimate, which is needed to prove Theorem \ref{thm:uniruling}:
$$E(u):= \int \omega(\partial_s u, J(s,t,u) \partial_s u) dsdt \leq \omega(u) - \int H_t(\gamma_+(t))dt + \int \max_M H_t dt.$$

This is a chain-map, by arguments similar to the ones considered in the previous sections, as long as $J(s,t) \subset \jreg$ for every $(s,t)$ and the family $J(s,t)$ is generic with $J(-\infty, 0) = J$.  It is standard to show that such choices are always possible.

There is a similar PSS$^{-1}$-morphism defined using pearls starting from an orbit $\tilde{\gamma}_-$ and going into a critical point using a half-disc $u$ satisfying the PSS$^{-1}$-equation
$$\partial_s u + J(s,t, u(s,t)) ( \partial_t u - \beta(-s) X_{H_t}(u)) = 0.$$
We count each homology class $A = \sum_i [v_i] + [u_{\gamma_-} \# u]$.  The energy of such half-discs verifies
$$
E(u) \leq \omega(u) + \int H_t(\gamma_-(t))dt - \int \min_M H_t dt.
$$

Recall that a chain-homotopy satisfying $d_{\text{pearl}} \psi - \psi d_{\text{pearl}} = \id - \text{PSS}^{-1} \circ \text{PSS}$ is defined by using pearls for which one of the discs $u$ solves the equation
$$\partial_s u + J(s,t, u(s,t))(\partial_t u + \alpha_R(s) X_{H_t}) = 0.$$
Here $R \in (0, \infty)$ depends on $u$, $\alpha_R\co \R \to [0,1]$ is smooth and, when $R \geq 1$, we set
$$
\alpha_R(s) = 
\begin{cases}
1 \text{ if } s \in [-R, R], \; \\
0 \text{ if } |s| \geq R + 1
\end{cases}.
$$
We require $|\alpha'_R| \leq 1$ and set $\alpha_R = R \alpha_1$ when $R \leq 1$.  The energy bound for such discs is given by
\begin{align}\label{eq:energypsshomotopy}
E(u) \leq \omega(u) + \int (\max_M H_t - \min_M H_t) dt.
\end{align}

Finally, there is a chain homotopy
$$\eta\co (\mathcal{C}(f, J) \otimes_\Lambda \mathcal{C}(g, J))_* \to CF_{*-1}(L, H, \mathbf{J})$$
satisfying 
$$
\text{PSS}(x \circ y) = x \star \text{PSS}(y) + (\partial \eta - \eta \partial)(x \otimes y),
$$
thus the PSS is a $QH(L)$-module isomorphism.
% !TEX root = lagsurf.tex
\section{A uniruling result and the proof of Theorem \ref{thm:main}}\label{sec:proof}
All the structures defined in \S\ref{sec:structures} as well as the energy estimates given there yield a proof of the following uniruling result, which can be found in \cite{Cha:refinement}; the modifications needed using the Novikov ring $\Lambda$ are direct:

\begin{thm}\label{thm:uniruling}
Let $L$ be an orientable Lagrangian surface and $H$ a non-constant Hamiltonian.  Then for every $J_0 \in \jreg$, $\mathbf{J} \subset \jreg$ starting at $J_0$ and every $x_0 \in L \backslash (\phi_1^H)^{-1}(L)$, there exists a non-constant map $u$ that is either a Floer $\mathbf{J}$-strip corresponding to $H$,  a $J_t$-holomorphic sphere or a $J_0$-holomorphic disc with boundary on $L$, such that $x_0 \in \text{Im}(u)$ and $0 < E(u) \leq \int (\max H_t -\min H_t)$.  If $L$ and $(\phi_1^H)^{-1}(L)$ intersect transversally, then $\mu(u) \leq 2$.
\end{thm}

\noindent \textbf{Proof of Theorem \ref{thm:main}}.
Let $L$ be a closed, orientable and displaceable Lagrangian surface, $E(L)$ its displacement energy, and $H$ a Hamiltonian displacing $L$ chosen such that $\int (\max_M H_t -\min_M H_t)dt = E(L) + \epsilon$ .  Given a symplectic embedding $e\co (B^{4}(r), \R^2) \to (M, L)$, the previous theorem shows that a non-constant $J$-holomorphic sphere (the constant path $\mathbf{J} = J$ is generic since there are no strips) or $J$-holomorphic disc with boundary on $L$, of symplectic area at most $E(L) + \epsilon$ goes through $e(0)$, where $e^*J = J_0$; here $J_0$ denotes the standard complex structure on $\C^2$.  Standard arguments (see e.g. Barraud-Cornea \cite[Proof of Corollary 3.10]{Bar-Cor:Serre}) then yield $\pi r^2 /2 \leq E(L) + \epsilon$ for every $\epsilon$. 
\qed

%%%%%%%%%%%%%%%%%%%%   End of main body of article
%
%                             References
%
%   BiBTeX users uncomment the following line:
%
\bibliographystyle{gtart}
\bibliography{bibliography}
\end{document}